\documentclass[11pt, reqno, a4paper]{amsart}

\usepackage[utf8]{inputenc}

\usepackage{amsmath}
\usepackage{amssymb}
\usepackage{amsthm}
\usepackage{amscd}
\usepackage{mathrsfs}
\usepackage{mathtools}
\usepackage{mathabx}

\newcommand{\vk}{\mathbf{k}}
\newcommand{\vc}{\mathbf{c}}
\newcommand{\vx}{\mathbf{x}}
\newcommand{\vb}{\mathbf{b}}
\newcommand{\vy}{\mathbf{y}}
\newcommand{\vnil}{\mathbf{0}}
\newcommand{\supp}{\mathrm{supp}}

\newcommand{\N}{\mathbb{N}}
\newcommand{\R}{\mathbb{R}}
\newcommand{\Z}{\mathbb{Z}}

\newcommand{\eps}{\epsilon}
\newcommand{\md}[1]{\ (\text{mod}\ #1)}
\newcommand{\ol}{\overline}
\newcommand{\starsum}[1]{\sideset{}{{}^*}\sum_{#1}}

\newtheorem{thm}{Theorem}[section]
\newtheorem{lem}[thm]{Lemma}
\newtheorem{cor}[thm]{Corollary}
\newtheorem{prop}[thm]{Proposition}

\theoremstyle{definition}
\newtheorem*{rem}{Remark}
\newtheorem{ex}{Example}

\usepackage{enumerate}

\usepackage[margin=1.5in]{geometry}

\numberwithin{equation}{section}

\begin{document}
\title{Weak approximation results for quadratic forms in four variables}
\author{Sofia Lindqvist\\Mathematical Institute, Oxford}
\address{Mathematical Institute, University of Oxford, Woodstock
Rd, Oxford OX2 6GG, UK}
\email{sofia.lindqvist@maths.ox.ac.uk}
\date{\today}

\begingroup
\let\MakeUppercase\relax 
\maketitle
\endgroup
\begin{abstract}
Let $F$ be a quadratic form in four variables, let $m\in\N$ and let $\vk\in \Z^4$. We count integer solutions to $F(\vx)=0$ with $\vx\equiv \vk\md{m}$. One can compare this to the similar problem of counting solutions to $F(\vx)=0$ without the congruence condition. It turns out that adding the congruence condition sometimes gives a very different main term than the homogeneous case. In particular, there are examples where the number of primitive solutions to the problem is $0$, while the number of unrestricted solutions is nonzero.
\end{abstract}

\section{Introduction}
Let $F$ be a non-singular quadratic form in $4$ variables with integer coefficients, and let $m\in \N,\vk\in \Z^4$ satisfy $F(\vk)\equiv 0 \md{m}$. We are interested in counting integer solutions to
\begin{equation}
F(\vx) = 0,\quad\vx \equiv \vk \md{m}
\label{eq:problem}
\end{equation}
inside a box of width $P$, as $P\to \infty$. We also assume that $(m,\vk)=1$, as otherwise one could just cancel the common factor everywhere. 

The same problem without the congruence condition $\vx\equiv\vk\md{m}$ was solved in \cite{heathbrown96}, and we will adopt the same methods here. We will therefore count solutions with a smooth weight $w:\R^4\to [0,\infty)$ with compact support. Fix a quadratic form $F$ and let
\[ N_{F,w}(P,m,\vk) = \sum_{\substack{\vx\in\Z^4\\\vx\equiv\vk \md{m}\\ F(\vx) =0}} w(P^{-1}\vx).\]
As in \cite{heathbrown96} we will require that $\nabla F \ne \vnil$ on the closure of $\supp(w)$. In particular this implies that $\supp(w)$ does not contain the origin. From now on we assume that this technical condition is satisfied without further comment.

For the corresponding quantity
\[ N_{F,w}(P) = \sum_{\substack{\vx\in \Z^4\\ F(\vx)=0}} w(P^{-1}\vx),\]
Heath-Brown obtains the following result.
\begin{thm}[{\cite[theorems 6 \& 7]{heathbrown96}}]\label{thm:hb?}
Let $F$ be a non-singular quadratic form in $4$ variables and let $w$ be a smooth weight function with compact support. If the determinant of $F$ is not a square it holds that
\[N_{F,w}(P) = \sigma_\infty(F,w) L(1,\chi) \sigma^*(F) P^2 + O_{F,w,\eps}(P^{3/2+\eps}) , \]
where $\chi$ is the Jacobi symbol $\big( \frac{\det(F)}{\cdot} \big)$. 

If the determinant of $F$ is a square it holds that
\[N_{F,w}(P) = \sigma_\infty(F,w) \sigma^*_\text{sq}(F)P^2\log P + \sigma_1(F,w)P^2+ O_{F,w,\eps}(P^{3/2+\eps}),  \]
where $\sigma_1(F,w)$ is some quantity not depending on $P$.
\end{thm}
Here $\sigma_\infty(F,w)$ denotes the singular integral
\begin{equation}
\sigma_\infty(F,w) = \lim_{\eps\to 0} \frac{1}{2\eps}\int_{|F(\vx)|\le \eps} w(\vx) d\vx,
\label{eq:sigmainfty}
\end{equation}
which can be shown to be positive if $w(\vx)>0$ for some real solution $\vx$ to $F(\vx)=0$. For quadratic forms in fewer than five variables the usual Hardy--Littlewood singular series may converge only conditionally or not at all, which is why the above main terms have a modified singular series $\sigma^*(F)$ or $\sigma^*_\text{sq}(F)$, which we define later. In exactly the same way as for the usual singular series one can show that these are positive provided $F(\vx)=0$ has a nonzero $p$-adic solution for every prime $p$.

Naively one would perhaps expect $N_{F,w}(P;m,\vk)$ to have a similar looking asymptotic expression as $N_{F,w}(P)$. Defining $F'(\vx) = F(m\vx+\vk)$ and $w'(\vx) = w(m\vx+\vk)$ it seems reasonable to expect $N_{F,w}(P;m,\vk)$ to have the main term predicted by Theorem \ref{thm:hb?} for $N_{F',w'}(P)$, keeping in mind that the theorem is not applicable in this case, as $F'$ is not a quadratic form. Rather surprisingly this turns out to not always be the case.

As an example of a case where the above prediction does not give the correct answer, consider the following. Let $p,q$ be odd primes satisfying $p\equiv q\equiv 1 \md{8}$ and set
\[F(\vx) = x_1^2-pqx_2^2-x_3 x_4.\]
Let $k$ satisfy $\left( \frac{k}{p} \right)=1$ and $\left( \frac{k}{q} \right)=-1$. We set $m=pq$ and $\vk = (a,b,k,k)$, where $a,b$ are any integers satisfying $F(\vk)\equiv 0\md{pq}$, e.g. $a=k,b=0$. We shall show in Section \ref{sec:ex} that then in fact
\[ N_{F,w}(P;m,\vk) \sim C P^2  \Big(1-\frac{L(2,\chi_\Delta)}{L(2,\chi_0)}\Big),\]
where $C$ is the leading coefficient predicted by Theorem \ref{thm:hb?}, $\chi_0$ is the principal Dirichlet character of conductor $pq$ and $\chi_\Delta = \left( \frac{\cdot}{pq} \right)$.

To understand why the above example has the ``wrong'' number of solutions we need to look at primitive solutions. We say that a vector $\vx$ is \emph{primitive} if the coordinates of $\vx$ have no common factor. Let
\[ N_{F,w}^\text{prim}(P,m,\vk) = \sum_{\substack{ \vx \text{ primitive}\\\vx\equiv \vk\md{m}\\F(\vx)=0}}w(P^{-1}\vx).\]
The quantities $N_{F,w}(P,m,\vk)$ and $N_{F,w}^\text{prim}(P,m,\vk)$ are then related via the formulae
\begin{equation}
N_{F,w}(P,m,\vk) = \sum_{\substack{d=1\\(d,m)=1}}^\infty N_{F,w}^\text{prim}(P/d,m,\ol{d}^{(m)} \vk)
\label{eq:n-normal}
\end{equation}
and
\begin{equation}
N_{F,w}^\text{prim}(P,m,\vk) = \sum_{\substack{d=1\\(d,m)=1}}^\infty \mu(d) N_{F,w}(P/d,m,\ol{d}^{(m)}\vk),
\label{eq:n-prim}
\end{equation}
where $\ol{d}^{(m)}$ is the inverse of $d$ modulo $m$.

Assume that $\vx$ is a primitive solution to \eqref{eq:problem} for the above example. Let $r|x_3$ be an odd prime. Then  $x_1^2-pqx_2^2\equiv 0\md{r}$, and so either $\left( \frac{pq}{r} \right)=\left( \frac{r}{pq} \right) = 1$ or $r|x_1$ and $r|x_2$. In the second case, we can't have $r=p$ or $r=q$, and since $\vx$ is primitive we then get $r^2|x_3$. Repeating this argument we see that any odd prime factor of $x_3$ either appears as an even power or is a quadratic residue mod $pq$. Noting that $\left( \frac{2}{pq} \right)=1$ and $\left( \frac{-1}{pq} \right)=1$ this then gives that in fact $\left( \frac{x_3}{pq} \right) = 1$, a contradiction. This shows that the example has no primitive solutions. By \eqref{eq:n-normal} we then expect that also the number of unrestricted solutions is biased, which is exactly what we observe.

Unfortunately we are not able to fully understand when the above phenomenon occurs. Let $\Delta$ be the conductor of the Dirichlet character $\chi_\Delta = \left( \frac{\det F}{\cdot} \right)$. What we are able to say is that $N_{F,w}(P;m,\vk)$ takes the expected form if $\det M$ is a square or if $\Delta$ has some prime factor that doesn't divide $m$. On the other hand, if all prime factors of $\Delta$ divide $m$ then there are two terms of order $P^2$ in the asymptotic expression for $N_{F,w}(P;m,\vk)$, namely the usual one and some sort of bias term. 

The bias term appears in the analysis in a similar way as the term $\sigma_1(F,w)P^2$ in Theorem \ref{thm:hb?} for the square determinant case. The quantity $\sigma_1(F,w)$ is not well understood, to the extent that there doesn't even seem to be a good way of predicting if it is nonzero. It is therefore not too surprising that it is hard to say anything quantitative about the bias term in our case. On the other hand, we are able to extract some qualitative information, so that knowing $N_{F,w}(P;m,\vk)$ or $N_{F,w}^\text{prim}(P;m,\vk)$ for some value of $\vk$ is enough to compute $N_{F,w}(P;m,d\vk)$ and $N_{F,w}^\text{prim}(P;m,d\vk)$ for any multiple $d\vk$ of $\vk$.
\begin{thm}\label{thm:main}
Let $F$ be a non-singular quadratic form in $4$ variables with underlying matrix $M$, let $m\in \N$ and $\vk\in \Z^4$ be such that $F(\vk)\equiv 0 \md{m}$, and let $w$ be a smooth weight function. Then, if $\det M$ is a square it holds that
\[
\begin{split}
N_{F,w}(P;m,\vk) = &\sigma_\infty(F,w) \sigma^*_\text{sq}(F,m,\vk) m^{-4}P^2\log P \\&+ \sigma_1(F,w,m,\vk)P^2 + O_{F,w,m,\eps}(P^{3/2+\eps})
\end{split}
\]
for any $\eps>0$. Here $\sigma_\infty(F,w)$ is as in \eqref{eq:sigmainfty} and $\sigma^*_\text{sq}(F,m,\vk)$ is the singular series, the definition of which will be given \eqref{eq:sigma-star}. The coefficient $\sigma_1(F,w,m,\vk)$ is some constant not depending on $P$.
\end{thm}

\begin{thm}\label{thm:main2}
Let $F$ be a non-singular quadratic form in $4$ variables with underlying matrix $M$, let $m\in \N$ and $\vk\in \Z^4$ be such that $F(\vk)\equiv 0 \md{m}$, and let $w$ be a smooth weight function. Then if $\det M$ is not a square it holds that, for any $\eps>0$,
\[
\begin{split}
N_{F,w}(P;m,\vk) =  \big(\sigma_\infty(F,w)L(1,\chi_\Delta)\sigma^*(F,m,\vk) + \tau(F,w,m,\vk)\big)m^{-4}P^2\\ + O_{F,w,\eps,m}(P^{5/3+\eps}),
\end{split}
\]
where we postpone the definition of the singular series $\sigma^*(F,m,\vk)$ until \eqref{eq:sigma-star2}, $\chi_\Delta = \left( \frac{\det M}{\cdot} \right)$, and $\tau(F,w,m,\vk)$ is some quantity not depending on $P$ that satisfies the following two properties.
\begin{itemize}
\item $\tau(F,w,m,\vk) =0$ unless every prime factor of $\Delta$ divides $m$;
\item If $(d,m)=1$ then $\tau(F,w,m,d\vk) = \chi_\Delta(d) \tau(F,w,m,\vk)$.
\end{itemize}

\end{thm}

Unfortunately we are not able to say much else about the quantity $\tau(F,w,m,\vk)$ than the two properties listed in the theorem. These do however immediately imply the following corollary, which in particular tells us that in projective space everything works out as expected.
\begin{cor}\label{thm:cor}
Let $F,w,m$ and $\vk$ be as in Theorem \ref{thm:main2} and assume that every prime factor of $\Delta$ divides $m$. Then if $\chi_\Delta(d)=1$ it holds that, for any $\eps>0$,
\[N_{F,w}(P;m,\vk) - N_{F,w}(P;m,d\vk) = O_{F,w,m,\eps}(P^{5/3+\eps}).\]

Furthermore, it holds that
\[
\begin{split}
\sum_{\substack{d\md{m}\\ (d,m)=1}} N_{F,w}(P;m,d\vk) = \sigma_\infty(F,w)L(1,\chi_\Delta)\sigma^*(F,m,\vk)\phi(m)m^{-4}P^2 \\+ O_{F,w,\eps,m}(P^{5/3+\eps})
\end{split}
\]
for any $\eps>0$.
\end{cor}

The error terms in Theorem \ref{thm:main2} and Corollary \ref{thm:cor} are both worse than the error terms in theorems \ref{thm:hb?} and \ref{thm:main}. This is not because of some actual limitation of the method, indeed one could obtain the same error term in all results by some extra work. However, allowing for a loss of $P^{1/6}$ in the error greatly simplifies the amount of technical work needed, which is why we have chosen to go for a slightly weaker approach here.

We remark that the problem we are trying to solve was mentioned in \cite[discussion surrounding Corollary 3]{heathbrown96}. There Heath-Brown points out that the counting function $N_{F,w}(P;m,\vk)$ may be handled with the same methods as those applied to $N_{F,w}(P)$, which is exactly what we do. On the other hand, the conclusion is perhaps not as straightforward as indicated. In particular, it is not true that excluding the obvious obstructions is enough to guarantee primitive solutions, as our initial example showed.

\subsection*{Notation}
For a $4\times 4$ matrix $A$ and a vector $\vx \in \R^4$ we'll write $A^{-1}(\vx) = \vx^T A^{-1} \vx$. The matrix $M$ will always be the symmetric matrix satisfying $F(\vx) = \vx^T M\vx$ and $\Delta$ will always be the conductor of the Dirichlet character $\left( \frac{\det M}{\cdot} \right)$, which we denote by $\chi_\Delta$. We will write $\chi_0$ for the principal Dirichlet character modulo $m$.

If $(a,b)=1$ we will write $\ol{a}^{(b)}$ for the inverse of $a$ modulo $b$, that is, $\ol{a}^{(b)} a\equiv 1 \md{b}$. The notation $a|b^\infty$ means that if $p|a$ then $p|b$.

We will write
\[\starsum{a\md{q}} = \sum_{\substack{a\md{q}\\ (a,q)=1}}\]
and in general we won't always specify all summation and integration limits if this is clear from context.

For a vector $\vx$ we write $|\vx|$ for the supremum norm $\|\vx\|_\infty$.

We use standard error notation. If $f$ and $g$ are two functions we write $f = O(g), f\ll g$ or $g\gg f$ to mean that there exists some absolute constant $C>0$ such that $f\le Cg$. Quite often this constant will also depend on some other parameters, which we indicate by a subscript. So $f=O_\eps(g)$ and $f\ll_\eps g$ both mean that there exists some constant $C(\eps)>0$ depending only on $\eps$ such that $f\le C(\eps) g$. 

\subsection*{Outline}
In section \ref{sec:hb} we apply the Heath-Brown circle method as described in \cite{heathbrown96} to get an expression for $N_{F,w}(P;m,\vk)$.  Most of the analysis is nearly identical as in \cite{heathbrown96}, and the parts which require some extra care are dealt with in Section \ref{sec:sq}. In Section \ref{sec:proof} we sketch how to prove Theorem \ref{thm:main}, Theorem \ref{thm:main2} and Corollary \ref{thm:cor}. Finally, in Section \ref{sec:ex} we revisit the example mentioned in the introduction together with a second example, and compute $N_{F,w}(P;m,\vk)$ for these two cases.

\subsection*{Acknowledgements} The author would like to thank Roger Heath-Brown for many valuable discussions and suggestions.

This work was partially supported by a grant from the Simons Foundation (award
number 376201 to Ben Green), and Ben Green's ERC starting Grant AAS 279438.
\section{The Heath-Brown circle method}\label{sec:hb}
The starting point of the Heath-Brown circle method is the following, which follows from an identity for the $\delta$-function due to Duke, Friedlander and Iwaniec \cite{dfi}. The method is therefore also sometimes referred to as the $\delta$-method.
\begin{thm}[\cite{heathbrown96} Theorem 2]
Let $F$ be a polynomial in four variables. For any $P>1$ there is a positive constant $c_P$ and an infinitely differentiable function $h(x,y)$ defined on the set $(0,\infty)\times \R$, such that
\[ \sum_{\vx\in \Z^4: F(\vx)=0} w(P^{-1}\vx)  = c_P P^{-2} \sum_{\vc\in\Z^4} \sum_{q=1}^\infty q^{-4} S_q(\vc) I_q(\vc), \]
where
\begin{equation}
S_q(\vc) = \starsum{a\md{q}} \sum_{\vb \md{q}} e_q(aF(\vb)+\vc\cdot\vb)
\label{eq:sq}
\end{equation}
and
\begin{equation}
I_q(\vc) = \int_{\R^4} w(P^{-1}\vx) h\left( \frac{q}{P},\frac{F(\vx)}{P^2} \right) e_q(-\vc\cdot\vx)d\vx.
\label{eq:iq}
\end{equation}
The constant $c_P$ satisfies $c_P = 1 + O_N(P^{-N})$ for any $N>0$.
\label{thm:cm}
\end{thm}

Note that in the theorem's original form $F$ is a form of degree $d$, but the homogeneity of $F$ isn't actually used in the proof. In addition much more can be said about the function $h$, but instead of stating these properties we will directly cite results from \cite{heathbrown96} involving $I_q(\vc)$. In particular we will use that there is a constant $C>0$ such that $I_q(\vc) =0$ for all $q\ge C P$.

Applying Theorem \ref{thm:cm} with $F(m\vx+\vk)$ and $w(P^{-1}(m\vx+\vk))$ in place of $F(\vx)$ and $w(P^{-1}(\vx))$ respectively, we get
\begin{equation}
N_{F,w}(P,m,\vk) = c_P P^{-2} \sum_\vc \sum_{q=1}^\infty q^{-4} S_q(\vc;m,\vk) I_q(\vc;m,\vk)
\label{eq:count}
\end{equation}
with
\begin{equation}
S_q(\vc;m,\vk) = \sideset{}{{}^*}\sum_{a \md{q}} \sum_{\vb \md{q}} e_q(aF(m\vb+\vk)+\vc\cdot\vb) 
\label{eq:sq_new}
\end{equation}
and
\[ I_q(\vc;m,\vk) = \int w(P^{-1}(m\vx+\vk)) h(P^{-1}q,P^{-2}F(m\vx+\vk)) e_q(-\vc\cdot\vx) d\vx.\]
Substitute $\vy = m\vx + \vk$ in the above integral to get
\begin{equation}
  I_q(\vc;m,\vk) = m^{-4} e_{mq}(\vc\cdot \vk) I_q(\vc/m),
\label{eq:Iq_new}
\end{equation}
where $I_q(\vc)$ is defined as in \eqref{eq:iq}. 

The fact that $\vc\in \Z^4$ rather than just $\R^4$ isn't needed to prove any of the results from \cite{heathbrown96} concerning $I_q(\vc)$, so all of the following results apply with $\frac{1}{m}\vc$ in place of $\vc$.  We will need these bounds when we combine everything in Section \ref{sec:proof}.
\begin{lem}[{\cite[Lemma 13]{heathbrown96}} ]
For all $N>0$ and $q\ll P$ it holds that
\[ I_q(\vnil) = P^4\big(\sigma_\infty(F,w)+O_{F,w,N}((q/P)^N)\big),\]
where $\sigma_\infty(F,w)$ is as in \eqref{eq:sigmainfty}.
\label{hb:13}
\end{lem}

\begin{lem}[{\cite[Lemma 16]{heathbrown96}}]\label{lem:hb16}
It holds that
\[ \frac{\partial^j I_q(\vnil)}{\partial q^j}\ll_{F,w}P^4q^{-j}\]
for $j=0,1$.
\end{lem}

\begin{lem}[{\cite[Lemma 19]{heathbrown96}}]
For $\vc\ne \vnil$ we have
\[I_q(\vc)\ll_{F,w,N} P^{5} q^{-1} |\vc|^{-N}\]
for any $N>0$.
\label{hb:19}
\end{lem}

\begin{lem}[{\cite[Lemma 22]{heathbrown96}}]
For $j=0,1$ it holds that
\[ \frac{\partial^j I_q(\vc)}{\partial q^j} \ll_{F,w,\eps} P^{3+\eps} q^{1-j}|\vc|^{-1}.\]
\label{hb:22}
\end{lem}

The main work needed to establish a result is to deal with the sums $S_q(\vc;m,\vk)$, which is what we do in the next section.

\section{The sum $S_q(\vc;m,\vk)$}\label{sec:sq}

We seek to prove an asymptotic for the sum
\begin{equation}
  S(X,\vc,m,\vk) = \sum_{q\le X} S_q(\vc;m,\vk) e_{mq}(\vc\cdot \vk),
\label{eq:S}
\end{equation}
where the factor $e_{mq}(\vc\cdot\vk)$ is included because of \eqref{eq:Iq_new}.


For a fixed value of $q$ write $q=uv$, where $(m,u)=1$ and $v|m^\infty$. Let $a = v a_u +u  a_v$ and $\vb = u\ol{u}^{(v)} \vb_v + v\ol{v}^{(u)}\vb_u$ in \eqref{eq:sq_new} to get that
\begin{equation}
S_{uv}(\vc;m,\vk) = S_u(\ol{v}^{(u)}\vc;m,\vk) S_v(\ol{u}^{(v)}\vc;m,\vk).
\label{eq:suv}
\end{equation}
Note that if $\vc = \vnil$ this simplifies to $S_q(\vnil;m,\vk)=S_u(\vnil;m,\vk) S_v(\vnil;m,\vk)$. In other words, $S_q(\vnil;m,\vk)$ is multiplicative in $q$, which will make this case much easier than the general case. 
\subsection*{The case $\vc=\vnil$}
As noted above $S_q(\vnil;m,\vk)$ is multiplicative as a function of $q$, and so we can use the the exact same methods as in \cite{heathbrown96} to understand the sum $\sum_{q\le X}q^{-4} S_q(\vnil;m,\vk)$. The resulting main term will involve the quantity 
\begin{equation}
\sigma_p(m,\vk) \coloneqq \sum_{s\ge 0} p^{-4s} S_{p^s}(\vnil;m,\vk).
\label{eq:sigma-new}
\end{equation}
In much the same way as one can show that the usual $p$-adic density
\begin{equation}
\sigma_p = \sum_{s\ge 0} p^{-4s} S_{p^s}(\vnil)
\label{eq:sigmap}
\end{equation}
satisfies
\[\sigma_p = \lim_{\nu \to \infty} p^{-3\nu} \#\{\vx\in (\Z/p^\nu\Z)^4: F(\vx) \equiv 0\md{p^{\nu}}\},\]
one can show that
\begin{equation}
\sigma_p(m,\vk) = \lim_{\nu\to \infty} p^{-3\nu} \#\{\vx \in (\Z/p^{\nu}\Z)^4: F(m\vx+\vk)\equiv 0\md{p^\nu}\}.
\label{eq:sigmap-interp}
\end{equation}
If $(p,m)=1$ it is clear that $\sigma_p(m,\vk) = \sigma_p$, for example one can substitute $m\vx+\vk\mapsto \vx$ in the above. 

We remark that for any $d$ satisfying $(d,m)=1$,
\begin{equation}
\sigma_p(m,d\vk)=\sigma_p(m,\vk),
\label{eq:sigmap-kdep}
\end{equation}
and so in reality $\sigma_p(m,\vk)$ only depends on the projective vector $\vk$.

\begin{lem}\label{lem:c0}
If $\det M$ is a square then
\[ \sum_{q\le X} q^{-4} S_q(\vnil;m,\vk) = \sigma^*_\text{sq}(F,m,\vk)\log X +\sigma_1'(F,m,\vk)+O_{F,m,\eps}(X^{-1/2+\eps}),\]
where
\begin{equation}
\sigma^*_\text{sq}(F,m,\vk) = \big(\prod_{p\notdivides m} (1-p^{-1}) \sigma_p\big) \big(\prod_{p|m} (1-p^{-1})\sigma_p(m,\vk)\big),
\label{eq:sigma-star}
\end{equation}
$\sigma_p(m,\vk)$ and $\sigma_p$ are as in \eqref{eq:sigma-new} and \eqref{eq:sigmap} respectively
and $\sigma_1'(F,m,\vk)$ is some constant depending on $F$, $m$ and $\vk$ but not on $X$.

If $\det M$ is not a square then
\[ \sum_{q\le X} q^{-4} S_q(\vnil;m,\vk) = L(1,\chi_\Delta)\sigma^*(F,m,\vk) +O_{F,m,\eps}(X^{-1/2+\eps}),\]
where
\begin{equation}
\sigma^*(F,m,\vk) =\big(\prod_{p\notdivides m}(1-\chi_\Delta(p)p^{-1})\sigma_p\big)\big(\prod_{p|m}(1-\chi_\Delta(p)p^{-1})\sigma_p(m,\vk)\big)
\label{eq:sigma-star2}
\end{equation}
and $\sigma_p(m,\vk)$ and $\sigma_p$ are as above.

\end{lem}
\begin{proof}
See \cite[Lemma 31]{heathbrown96}, which gives a similar statement with $S_q(\vnil)$ in place of $S_q(\vnil;m,\vk)$. The details are otherwise almost exactly the same, so we do not repeat them here.
\end{proof}

\subsection*{The case $\vc\ne\vnil$}
If $(u,m)=1$ it holds that
\[
  S_u(\ol{v}^{(u)}\vc;m,\vk) = e_u(-\ol{mv}^{(u)} \vc\cdot\vk) S_u(\ol{mv}^{(u)}\vc) = e_u(-\ol{mv}^{(u)}\vc\cdot\vk)S_u(\vc),
\]
by the substitution $\vb' = m\vb + \vk$ in \eqref{eq:sq_new} and the fact that $S_q(k\vc)=S_q(\vc)$ for $(q,k)=1$. Using this and \eqref{eq:suv} gives
\[
e_{uv}(\vc\cdot\vk/m)S_{uv}(\vc;m,\vk) = e_{mv}(\ol{u}^{(mv)}\vc\cdot\vk) S_u(\vc) S_v(\ol{u}^{(v)}\vc;m,\vk),
\]
where we also use that
\[  \ol{a}^{(b)} + \ol{b}^{(a)} \equiv 1 \md{ab} \]
to write $e_{muv}(\vc\cdot\vk) e_u(-\ol{mv}^{(u)}\vc\cdot\vk) = e_{mv}(\ol{u}^{(mv)}\vc\cdot\vk)$.

We will use this to write the sum over $q$ in \eqref{eq:S} as a sum over $u$ and $v$, where $(u,m)=1$ and $v|m^\infty$. We then split the sum into two parts depending on the size of $v$, namely $S(X,\vc,m,\vk) = S_1+S_2$ with
\[S_1 = \sum_{\substack{v|m^\infty\\v\le X^{1/3}}} \sum_{u\le X/v} e_{mv}(\ol{u}^{(mv)} \vc\cdot\vk) S_u(\vc) S_v(\ol{u}^{(v)}\vc;m,\vk)\]
and
\[S_2 = \sum_{\substack{v|m^\infty\\ X^{1/3}<v\le X}} \sum_{u\le X/v} e_{mv}(\ol{u}^{(mv)} \vc\cdot\vk) S_u(\vc) S_v(\ol{u}^{(v)}\vc;m,\vk).\]
For $S_2$ we use the following basic bound.
\begin{lem}\label{lem:sq-bound}
It holds that
\[ S_q(\vc,m,\vk) \ll_{m,\Delta} q^{3}.\]
\end{lem}
\begin{proof}
This is essentially {\cite[Lemma 25]{heathbrown96}}, which establishes the bound for $S_q(\vc)$, but pretty much the exact same proof goes through for $S_q(\vc,m,\vk)$.
\end{proof}
Taking absolute values and using the lemma gives
\[|S_2|\ll_{m,\Delta} \!\!\!\!\!\!\!\sum_{\substack{v|m^\infty\\ X^{1/3}<v\le X}}\!\!\! \sum_{u\le X/v} u^3v^3 \le \!\!\!\!\!\!\! \sum_{\substack{v|m^\infty\\ X^{1/3}<v\le X}}\!\!\! \sum_{u\le X/v} X^3 \le \sum_{\substack{v|m^\infty\\ v\le X}} \sum_{u\le X^{2/3}} X^3\ll_{\eps,m} X^{11/3+\eps}.\]

For a function $f$ supported only on integers coprime to $mv$ with period $mv$ one can decompose
\[ f(u) = \sum_{\chi \md{mv}} a_\chi \chi(u),\]
where
\[a_\chi = \frac{1}{\phi(mv)}\sum_{x\md{mv}} f(x) \ol{\chi(x)}.\]
Using this in $S_1$ we get
\[ S_1 = \sum_{v|m^\infty, v\le X^{1/3}} \sum_{\chi \md{mv}} a_\chi(v,\vc,m,\vk)\sum_{u\le X/v} \chi(u) S_u(\vc) \]
with
\begin{equation}
  a_\chi(v,\vc,m,\vk) = \frac{1}{\phi(mv)}\sum_{x\md{mv}} \chi(x) e_{mv}(x\vc\cdot\vk)S_v(x\vc;m,\vk).
\label{eq:achi}
\end{equation}
Dealing with the innermost sum over $u$ will be done nearly identically as in \cite{heathbrown96}, the only difference is the twist by $\chi(u)$.

\begin{lem} Let $\chi$ be a Dirichlet character of conductor $mv$ for some $v|m^\infty$. Then for any $\eps>0$,
\[ \sum_{u\le X} \chi(u) S_u(\vc) = \eta(\vc)\xi(\chi)\sigma'(F,m)\frac{X^4}{4}+ O_\eps(X^{7/2+\eps}|\vc|^\eps), \]
where
\begin{equation}
\sigma'(F,m) =\big(\prod_{p\notdivides m} (1-p^{-1})\sigma_p'\big)\big(\prod_{p|m}(1-p^{-1})\big),
\label{eq:sigma-dash}
\end{equation}
\begin{equation}
\sigma_p' = \sum_{t\ge 0} p^{-4t} S_{p^t}(\vnil) \chi_\Delta(p)^t,
\label{eq:sigmap-dash}
\end{equation}
$\eta(\vc)=1$ if $M^{-1}(\vc)=0$ and $0$ otherwise, and $\xi(\chi)=1$ if $\chi = \chi_0 \chi_\Delta$ and $0$ otherwise. Recall that $\chi_0$ is the principal character modulo $m$.
\label{lem:ssum-twisted}
\end{lem}
\begin{proof}
If $M^{-1}(\vc)\ne 0$ this is just \cite[Lemma 28]{heathbrown96} with the condition $|\vc|\le P$ removed and the error $P^\eps$ replaced by $|\vc|^\eps$.

If $M^{-1}(\vc)=0$ this is essentially the same as \cite[Lemma 30]{heathbrown96}. The only real difference is that if $\chi=\chi_0\chi_\Delta$ then one gets a main term even when $\det M$ is not a square. Note that in the main term we have used the fact that $S_{p^t}(\vc) = S_{p^t}(\vnil)$ whenever $M^{-1}(\vc)=0$ and $p\notdivides \Delta$. Indeed, 
\[aF(\vb)+\vc\cdot\vb = aF(\vb + (2a)^{-1}M^{-1}\vc) - (4a)^{-1}M^{-1}(\vc)\]
can be used in the definition of $S_{p^t}(\vc)$.
\end{proof}

Applying Lemma \ref{lem:ssum-twisted} we now have that $S_1$ is
\begin{equation*}
\sum_{v|m^\infty, v\le X^{1/3}} \sum_{\chi \md{mv}} a_\chi(v,\vc,m,\vk) \Big(\eta(\vc)\xi(\chi)\sigma'(F,m) \frac{X^4}{4v^4} + O_\eps( (X/v)^{7/2+\epsilon}|\vc|^\eps) \Big).
\end{equation*}
For the error term we use that 
\begin{equation}
|a_\chi(v,\vc,m,\vk)| \ll_{m,\Delta} v^3
\label{eq:achi-bound-triv}
\end{equation}
by \eqref{eq:achi} and the bound in Lemma \ref{lem:sq-bound}. The error is then bounded above by
\[ \sum_{v|m^\infty, v\le X^{1/3}} O_{m,\Delta,\eps} (v^{1/2-\eps}X^{7/2+\eps}|\vc|^\eps)\ll_{m,\Delta,\eps} X^{11/3+\eps}|\vc|^\eps,\]
where we have used that the number of values of $v$ we are summing over is of order $(\log X)^{\omega(m)}\ll_{m,\eps} X^\eps$, where $\omega(m)$ is the number of unique prime factors of $m$.

If $\eta(\vc)=1$ and $\Delta|m^\infty$ we might also get a main term. Assume therefore that $\Delta|m^\infty$, such that for at least some $v$ the character $\chi = \chi_0\chi_\Delta$ appears in the sum. The main term is then
\[\eta(\vc)\sigma'(F,m)\frac{X^4}{4}\sum_{\substack{v|m^\infty, v\le X^{1/3}\\ \Delta | mv}} a_{\chi_0\chi_\Delta}(v,\vc,m,\vk) v^{-4}. \]
Finally we extend the sum over $v$ to also include values with $v>X^{1/3}$. By \eqref{eq:achi-bound-triv} this introduces an error $O_{m,\Delta,\eps}(X^{11/3+\eps})$. Combining all of this we thus have
\[ S_1 = \frac{X^4}{4}\eta(\vc)\sigma'(F,m) A(\vc,m,\vk) + O_{m,\Delta,\eps}(X^{11/3+\eps}),\]
where we have defined
\begin{equation}
A(\vc,m,\vk) = \sum_{v|m^\infty, \Delta|mv} a_{\chi_0\chi_\Delta}(v,\vc,m,\vk) v^{-4}.
\label{eq:A}
\end{equation}
Together with the bound on $S_2$ we have thus proved the following.
\begin{lem}\label{lem:cnonzero}
For any $\eps>0$ and $\vc\ne \vnil$ it holds that
\[S(X,\vc,m,\vk) = \frac{X^4}{4}\eta(\vc)\sigma'(F,m)A(\vc,m,\vk) + O_{m,\Delta,\eps}(X^{11/3+\eps}|\vc|^\eps).\]
\end{lem}
\begin{rem}
One can in fact evaluate $a_\chi(v,\vc,m,\vk)$ by factoring $v$ and $m$ into separate prime factors and using standard results on Gauss sums. However, the resulting expressions don't seem very enlightening, and the dependence on $\vc$ is particularly nasty. In proving our main theorem we will eventually need to sum an expression involving $\eta(\vc)A(\vc,m,\vk)$ over $\vc\ne\vnil$, and here we were not able to use the closed form of $a_\chi(v,\vc,m,\vk)$ and $A(\vc,m,\vk)$ in any meaningful sense. Because computing a closed form expression is a fairly long and technical computation we thus omit it here.
\end{rem}

Finally, we record the following transformation property for $A(\vc,m,\vk)$.
\begin{lem}\label{lem:A-trans}
Let $(d,m)=1$. Then
\[ A(\vc,m,d\vk) = \chi_\Delta(d) A(\vc,m,\vk).\]
\end{lem}
\begin{proof}
By \eqref{eq:achi} and \eqref{eq:sq_new} we have that $a_\chi(v,\vc,m,\vk)$ is
\[
\sum_{x\md{mv}}\;\starsum{a\md{v}}\sum_{\vb\md{v}} \chi(x) e_{mv}(x\vc\cdot(m\vb+\vk)) e_{v}(aF(m\vb+\vk)).\]
If $(d,m)=1$ we can make the substitutions $d\vb\mapsto \vb$, $x\mapsto dx$ and $a\mapsto d^2 a$ to get $a_\chi(v,\vc,m,\vk) = \chi(d) a_\chi(v,\vc,m,d\vk)$.
Inserting this into \eqref{eq:A} then gives the result.
\end{proof}

\section{Proof of main results}\label{sec:proof}
We are now ready to prove our main results. Assume first that $\det M$ is a square. The proof of Theorem \ref{thm:main} is exactly the same as the proof of \cite[Theorem 7]{heathbrown96}, except that the main term comes from Lemma \ref{lem:c0} instead of an analogous result for $\sum_{q\le X} q^{-4}S_q(\vnil)$, and similarly one needs to use Lemma \ref{lem:cnonzero} instead of an analogous result for $\sum_{q\le X} S_q(\vc)$ to achieve the term of order $P^2$. Note that by \eqref{eq:achi-bound-triv}, $A(\vc,m,\vk)$ can be bounded above by $O_m(1)$ uniformly in $\vc$, and so showing convergence of the sum over $\vc\ne\vnil$ can be done exactly as before.

When $\det M$ is not a square the proof is essentially just a combination of the proofs of \cite[theorems 6 \& 7]{heathbrown96}. As usual $\vc=\vnil$ in \eqref{eq:count} gives rise to a term of order $P^2$, but unlike in Theorem \ref{thm:hb?} the terms with $\vc\ne\vnil$ also contribute something of order $P^2$, in the same way that one gets a secondary term of order $P^2$ in the case where $\det M$ is a square. As this is the most interesting case we repeat a rough sketch of the proofs in \cite{heathbrown96}.

\begin{proof}[Sketch of proof of Theorem \ref{thm:main2}]
Using lemmas \ref{lem:hb16} and \ref{lem:c0} and the identity \eqref{eq:Iq_new} we have by partial summation that
\[\sum_{R<q\le 2R} q^{-4}S_q(\vnil;m,\vk)I_q(\vnil;m,\vk)\ll_{F,w,m,\eps} P^{4}R^{-1/2+\eps}.\]
Since $I_q(\vnil)=0$ for $q\gg P$ this then gives that
\[
\sum_{q>P^{1-\eps}} q^{-4}S_q(\vnil;m,\vk)I_q(\vnil;m,\vk) \ll_{F,w,m,\eps} P^{7/2+2\eps}.
\]
For $q\le P^{1-\eps}$, lemmas \ref{hb:13} and \ref{lem:c0} give that
\[
\begin{split}
\sum_{q\le P^{1-\eps}} q^{-4}S_q(\vnil;m,\vk)I_q(\vnil;m,\vk) = \sigma_\infty(F,w)L(1,\chi_\Delta)\sigma^*(F,m,\vk)m^{-4}P^4 \\+ O_{F,w,m,\eps}(P^{7/2+\eps}).
\end{split}
\]

Lemmas \ref{hb:19} and \ref{lem:sq-bound} together with the fact that $I_q(\vc)$ is supported only for $q\ll P$ gives
\[ \sum_{|\vc|\ge P^\eps} \sum_{q=1}^\infty q^{-4}S_q(\vc;m,\vk) I_q(\vc;m,\vk) \ll_{F,w,m,\eps} 1.\]
For $|\vc|\le P^\eps$ and $\vc\ne \vnil$ we use partial summation together with Lemma \ref{hb:22} and Lemma \ref{lem:cnonzero} to get
\[
\begin{split}
\sum_{R<q\le 2R} &q^{-4} e_{mq}(\vc\cdot\vk) S_q(\vc;m,\vk) I_q(\vc/m) \\
&= \eta(\vc)\sigma'(F,m)A(\vc,m,\vk) \int_R^{2R} t^{-1} I_t(\vc/m)dt + O_{F,w,m,\eps}(P^{3+\eps}R^{2/3+\eps}).
\end{split}
\]
Summing over all $q$ and using that $I_q(\vc)=0$ for $q\gg P$ we then get
\[
\begin{split}
\sum_{q =1}^\infty &q^{-4} e_{mq}(\vc\cdot\vk) S_q(\vc;m,\vk) I_q(\vc/m) \\
&= \eta(\vc)\sigma'(F,m)A(\vc,m,\vk) \int_0^{\infty} t^{-1} I_t(\vc/m)dt + O_{F,w,m,\eps}(P^{11/3+2\eps})
\end{split}
\]
for any $\vc\ne \vnil$. Now we define
\[I_r^*(\vc) = P^{-4} I_{rP}(\vc),\]
which doesn't depend on $P$. The last integral above is then $P^4\sigma_\infty(F,w,\vc/m)$, where we have defined
\[ \sigma_\infty(F,w,\vc) = \int_0^\infty r^{-1}I_r^*(\vc) dr.\]
The convergence of this is shown in \cite{heathbrown96}, but we don't repeat the argument here. In fact it holds that
\[ \sigma_\infty(F,w,\vc) \ll_{F,w,N} |\vc|^{-N}\]
for any $N>0$. For the main term we may therefore extend the sum $\sum_{0<|\vc|\le P^\eps}$ to all $\vc\ne\vnil$, which upon plugging everything into \eqref{eq:count} finally gives us that
\[
\begin{split}
N_{F,w}(P;m,\vk) =& m^{-4}P^2\Big(\sigma'(F,m) \sum_{\vc \ne \vnil} \eta(\vc)A(\vc,m,\vk)\sigma_\infty(F,w,\vc/m) \\
&+ \sigma_\infty(F,w)L(1,\chi_\Delta) \sigma^*(F,m,\vk)\Big) + O_{\eps,F,w,m}(P^{5/3+\eps}),
\end{split}\]
where we have redefined $\eps$. By defining
\[\tau(F,w,m,\vk) = \sigma'(F,m)\sum_{\vc\ne\vnil} \eta(\vc)A(\vc,m,\vk)\sigma_\infty(F,w,\vc/m)\]
we then have the required asymptotic expression.

Finally, $A(\vc,m,\vk)=0$ whenever $\Delta \notdivides m^\infty$, and so the same holds for $\tau(F,w,m,\vk)$. Lemma \ref{lem:A-trans} gives that $\tau(F,w,m,d\vk)= \chi_\Delta(d)\tau(F,w,m,\vk)$ for any $d$ satisfying $(m,d)=1$.
\end{proof}

\begin{proof}[Proof of Corollary \ref{thm:cor}]
For the first part of the corollary we just need that for $(m,d)=1$ it holds that $\sigma^*(F,m,d\vk)=\sigma^*(F,m,\vk)$ by \eqref{eq:sigmap-kdep} together with $\tau(F,w,m,d\vk)=\chi_\Delta(d)\tau(F,w,m,\vk)$.

For the second part of the corollary we simply use \[\tau(F,w,m,d\vk)= \chi_\Delta(d)\tau(F,w,m,\vk),\] together with the fact that $\chi_\Delta(d) =-1$ for exactly half of the values $d\in(\Z/m\Z)^*$ when $\Delta|m^\infty$.  Here we also use that $\Delta\ne 1,2$, as for squarefree $D$ the quadratic character $\left( \frac{D}{\cdot} \right)$ has conductor $|D|$ or $4|D|$, and so $\Delta =2$ will never occur.
\end{proof}

\begin{rem}
Now that we have the precise definition of $\tau(F,w,m,\vk)$ it should be clear what causes the difficulty in analysing it. In particular we see that in addition to understanding $A(\vc,m,\vk)$ one needs to say something more about $\sigma_\infty(F,w,\vc)$, a feat that also would tell us something about $\sigma_1(F,w)$ in Theorem \ref{thm:hb?}. 
\end{rem}

\section{Examples}\label{sec:ex}
We begin by noting that the second result of Corollary \ref{thm:cor} implies a similar result for $N_{F,w}^\text{prim}(P;m,\vk)$. Indeed, by the corollary and \eqref{eq:n-prim} we get that
\[\begin{split}
\starsum{d\md{m}} N_{F,w}^\text{prim}(P;m,d\vk) &= \sum_{\substack{l=1\\(l,m)=1}}^\infty \mu(l) \starsum{d\md{m}} N_{F,w}(P/l; m,d\ol{l}^{(m)}\vk)\\
&= CP^2\sum_{(l,m)=1} \mu(l)l^{-2} + O_{F,w,m,\eps}\left(P^{5/3+\eps}\right)\\
&= \frac{C}{L(2,\chi_0)}P^2+O_{F,w,m,\eps}(P^{5/3+\eps}),
\end{split}\]
where $C =  \sigma_\infty(F,w) L(1,\chi_\Delta) \sigma^*(F,m,\vk) \phi(m)m^{-4}$. Note that we have used that $\sigma^*(F,m,\vk)$ is unchanged when $\vk$ is multiplied by $l\in (\Z/m\Z)^*$.

From the first part of Corollary \ref{thm:cor} together with \eqref{eq:n-prim} we also get that
\[N_{F,w}^\text{prim}(P;m,\vk) - N_{F,w}^\text{prim}(P;m,d\vk) = O_{F,w,m,\eps}(P^{5/3+\eps})\]
whenever $\chi_\Delta(d) = 1$.

If we know that $N_{F,w}^\text{prim}(P;m,\vk)=0$ for some value of $\vk$, as will be the case in both examples, we then get the following.
\begin{prop}
Assume that $\Delta|m^\infty$, $F$ has a non-square determinant and $N_{F,w}^\text{prim}(P;m,\vk)=0$. Then for $(d,m)=1$ and any $\eps>0$ it holds that
\[\begin{split}N_{F,w}^\text{prim}(P;m,d\vk)= &\begin{cases}
0, &\text{if } \chi_\Delta(d) = 1\\
\frac{2}{L(2,\chi_0)}\sigma_\infty(F,w) L(1,\chi_\Delta) \sigma^*(F,m,\vk) m^{-4}P^2, &\text{if } \chi_\Delta(d) = -1
\end{cases}\\&+ O_{F,w,m,\eps}(P^{5/3+\eps})\end{split}
\]
and
\[\begin{split}N_{F,w}(P;m,d\vk)=&
\sigma_\infty(F,w)L(1,\chi_\Delta)\sigma^*(F,m,\vk)m^{-4}P^2\left(1-\chi_\Delta(d)\frac{L(2,\chi_0\chi_\Delta)}{L(2,\chi_0)}\right)\\
&+ O_{F,w,m,\eps}(P^{5/3+\eps}).\end{split}
\]
\label{prop:prim}
\end{prop}
\begin{proof}
The first part follows by the above discussion and the fact that $\chi_\Delta(d) = 1$ for exactly half the values of $d\in (\Z/m\Z)^*$.

For the second part we use \eqref{eq:n-normal} to get
\[\begin{split}N_{F,w}(P;m,d\vk) &= \sum_{(l,m)=1} N_{F,w}^\text{prim}(P/l,m,d\ol{l}^{(m)}\vk)\\
&= C_1P^2\sum_{l: \chi_0\chi_\Delta(l)=1} \frac{1}{l^{2}}+C_2P^2\sum_{l:\chi_0\chi_\Delta(l)=-1}\frac{1}{l^2}+O_{F,w,m,\eps}(P^{5/3+\eps}),
\end{split}\]
where if $\chi_\Delta(d)=1$ we have $C_1=0$ and $C_2 = 2\sigma_\infty(F,w)L(1,\chi_\Delta)\sigma^*(F,m,\vk)m^{-4}$, and otherwise $C_1$ and $C_2$ swap roles. Finally,
\[\sum_{l:\chi_0\chi_\Delta(l)=1}l^{-2} = \frac{1}{2}\sum_{(l,m)=1} (\chi_\Delta(l)+1)l^{-2} = \frac{L(2,\chi_0\chi_\Delta)+L(2,\chi_0)}{2},\]
with a similar expression for $\sum_{l:\chi_0\chi_\Delta=-1}l^{-2}$, finishing the proof.
\end{proof}

\begin{ex}
As in the introduction, let
\[F(\vx) = x_1^2-pqx_2^2-x_3x_4,\]
let $p\equiv q\equiv 1\md{8}$ be odd primes and take $k$ such that $\left(\frac{k}{p}\right)=1$ and $\left(\frac{k}{q}\right)=-1$. Set $m=pq$ and $\vk =(k,0,k,k)$. As shown in the introduction we know that
\[N_{F,w}^\text{prim}(P;m,\vk) = 0,\]
and so Proposition \ref{prop:prim} us applicable.
\end{ex}

\begin{ex}
As a second example, set $m=4$, $\vk= (0,3,3,3)$ and
\[F(\vx) = x_1^2+x_2^2-x_3x_4.\]
Furthermore, let $w$ be a smooth weight function which satisfies $\supp(w)\subset (0,\infty)^4$. Assume that $\vx$ is a primitive solution to $F(\vx)=0$ with $\vx\equiv \vk\md{m}$ and let $p|x_3$ be an odd prime. Then $x_1^2+x_2^2\equiv 0\md{p}$, and as $x_1,x_2$ are positive this implies that $p\equiv 1\md{4}$ or $p|x_1$ and $p|x_2$. In the second case, as $\vx$ is primitive we get $p^2|x_3$. In this way one can show that any prime $p\equiv 3\md{4}$ divides $x_3$ to an even power, contradicting the fact that $x_3\equiv 3 \md{4}$, and so there are no primitive solutions. Proposition \ref{prop:prim} is thus applicable also in this case.
\end{ex}

\bibliographystyle{alpha}
\bibliography{literature}

\begin{thebibliography}{DFI93}

\bibitem[DFI93]{dfi}
W.~Duke, J.~Friedlander, and H.~Iwaniec.
\newblock Bounds for automorphic $l$-function.
\newblock {\em Invent. Math.}, 112:1--8, 1993.

\bibitem[HB96]{heathbrown96}
R.~Heath-Brown.
\newblock A new form of the circle method, and its application to quadratic
  forms.
\newblock {\em J. Reine Angew. Math.}, 481, 1996.

\end{thebibliography}

\end{document}